\documentclass[12pt,reqno]{amsart}

\usepackage{wasysym,latexsym}
\usepackage{mathtools,xcolor,enumitem}

\numberwithin{equation}{section}

\newtheorem{lemma}{Lemma}[section]
\newtheorem{theorem}{Theorem}[section]

\theoremstyle{definition}

 \DeclareMathOperator{\RE}{Re}

\begin{document}

\title[Starlikeness of Certain Analytic Functions]{Starlikeness of Certain Analytic Functions}

\dedicatory{Dedicated to Prof.\ Dato' Indera Rosihan M. Ali}

\author[A. S. Ahmad Al-faqeer]{Ahmad Sulaiman Ahmad El-faqeer}

\address{School of Mathematical Sciences,
Universiti Sains Malaysia, 11800 Penang, Malaysia}
\email{ahmedfakier@student.usm.my}

\author[M. H. Mohd]{Maisarah Haji Mohd }
\address{ School of Mathematical Sciences,
Universiti Sains Malaysia, 11800 Penang, Malaysia}
\email{ maisarah-hjmohd@usm.my}

\author{\break V. Ravichandran}
\address{Department of Mathematics, National Institute of Technology, Tiruchirappalli 620 015,
India} \email{ravic@nitt.edu; vravi68@gmail.com}

\author[S. Supramaniam]{Shamani Supramaniam}
\address{ School of Mathematical Sciences,
Universiti Sains Malaysia, 11800 Penang, Malaysia}
\email{shamani@usm.my}

\subjclass[2000]{30C45}

\begin{abstract}
Let $f$ and $g$ be analytic functions on the open unit disk of the complex plane with $f/g$ belonging to the class $\mathcal{P} $ of functions  with positive real part  consisting of functions $p$ with $p(0)=1$ and $\RE p(z)>0$ or to its subclass  consisting of functions $p$ with $|p(z)-1|<1$.
We obtain the sharp radius constants for the  function $f$  to be starlike   of order $\alpha$, parabolic starlike, etc. when  $g/k\in\mathcal{P}$
where $k$ denotes the Koebe function defined by $k(z)=z/(1-z)^2$.
\end{abstract}

\keywords{Starlikeness; radius problem; functions with positive real part; parabolic starlike functions; univalent functions}
\subjclass{30C45; 30C80}
\maketitle

\section{Motivation and Radii Results}
Let $\mathcal{A}$ be the class of all analytic function $f$ on the open unit disk $\mathbb{D}$ normalized by $f(0)=0$ and $f'(0)=1$, and let $\mathcal{S}$ be the subclass  of all univalent function in $\mathcal{A}$. For  two arbitrary subclasses  $\mathcal{F}$ and $\mathcal{G}$ of  $\mathcal{A}$, the $\mathcal{G}$ radius of $\mathcal{F}$, denoted by $\mathcal{R}_{\mathcal{G}}(\mathcal{F})$,  is defined as the largest number $\mathcal{R}_{\mathcal{G}}$ such that $r^{-1}f(rz)\in \mathcal{G}$ for  all $r$ with  $0<r<\mathcal{R}_{\mathcal{G}}$, and for all $f\in\mathcal{F}$. Whenever the class $\mathcal{G}$ is characterized by a geometric property $\mathbf{P}$ the number $\mathcal{R}_\mathcal{G}$ is called as the radius of the property $\mathbf{P}$ of the class $\mathcal{F}$. Although there are variety of  radius problems considered in literature, we investigate the functions $f$ characterized by the ratio of $f$ with another function $g\in\mathcal{A}$; these kinds of problems were considered by  MacGregor \cite{T.H.M,T.H.MM,T.H.MMM}. Ali \emph{et al.} determined various radii results for  functions $f$ satisfying the following conditions:
(i) $\operatorname{Re}\left(f(z)/g(z)\right)>0$ where $\operatorname{Re} \left(g(z)/z\right)>0$ or $\operatorname{Re} \left(g(z)/z\right)>1/2$ (ii)  $|(f(z)/g(z)-1)-1|<1$ where $\operatorname{Re} \left(f(z)/g(z)\right)>0$ or $g$ is convex. All these classes are associated to class of functions with positive real part; this class, denoted by $\mathcal{P}$, consists of   all analytic functions $p : \mathbb{D} \rightarrow \mathbb{C}$ with $p(0)=1$ and $\operatorname{Re}\left( p(z)\right)> 0$ for all $z \in \mathbb{D}$. Asha and Ravichandran \cite{AS.VR} investigated  several radii for the functions $f/g\in \mathcal{P}$ and $(1+z)g/z \in \mathcal{P}$, belonging to some subclasses of starlike functions (see \cite{Kanaga,Kanika} for further works). For $0\leq \alpha<1$, we let $\mathcal{P}(\alpha):=\{p\in\mathcal{P}: \operatorname{Re}p(z)>\alpha\}$.  Let  $k$ be the Koebe function defined by $k(z)=z/(1-z)^2$.  In this paper, we consider the two subclasses $\Pi_{1}$ and $\Pi_{2}$ of analytic functions given below:
\begin{align*}
\Pi_{1}& =: \left\{f\in \mathcal{A}: f/g\in\mathcal{P}     \text{ for some }   g\in \mathcal{A} \text{ with }   g/k\in\mathcal{P}   \right\},\shortintertext{and}
\Pi_{2}&=:\left\{f\in \mathcal{A}: \left|f/g-1\right| <1   \text{ for some }   g \in \mathcal{A} \text{ with }  g/k \in\mathcal{P}  \right\}.
\end{align*}
We determine radii for functions in these two classes to belongs to  several subclasses of starlike functions which we discuss below.

In $1985$, Padmanabhan and Parvatham \cite{KSP} used the Hadamard product (convolution) and subordination to introduce the class of functions  $f\in\mathcal{A}$ satisfying  $z(k_{\alpha}\ast f)'/(k_{\alpha}\ast f)\prec h$ where $k_{\alpha}(z)=z/(1-z)^{\alpha}$, $ \alpha\in \mathbb{R}$, and $h$ is convex. When $h$ is the normalized mapping of $\mathbb{D}$ onto the right half-plane, this class  reduces to the usual classes of starlike and convex functions respectively for $\alpha=1$ and $\alpha=2$. Later, in 1989, Shanmugam \cite{T.N.S} studied the class $\mathcal{S}^*_{g}(\varphi)=:\left\{f\in \mathcal{A}:z(f\ast g)'/(f\ast g)\prec \varphi\right\}$ where $g$ is fixed and $\varphi$ a convex function, respectively; this class includes several classes defined by means of linear operator such as Rucheweyh differential operator and  Salagean operator. When $g(z)=z/(1-z)$ and $g(z)=z/(1-z)^2$,  the subclass $\mathcal{S}^*_{g}(\varphi)$ is denoted respectively by  $\mathcal{S}^*(\varphi)$ and $\mathcal{K}(\varphi)$. In $1992$, Ma and Minda \cite{W.C.MA} studied growth, distortion, covering theorems and coefficient problems for the  classes  $\mathcal{S}^*(\varphi)$ and $\mathcal{K}(\varphi)$ when   $ \varphi\in \mathcal{P}$ is just a  univalent function mapping unit disk $\mathbb{D}$ onto domain symmetric with respect to the real line and starlike with respect to $\mathcal{\varphi}(0)=1$ and $\mathcal{\varphi}'(0)>0$. For $\varphi(z)=(1+(1-2\alpha)z)/(1-z)$ with $0\leq \alpha<1$, the classes $\mathcal{S}^*(\varphi)$ and $\mathcal{K}(\varphi)$ reduce  to the class $\mathcal{S}^*(\alpha)$ of starlike functions of  order $\alpha$ and the class $\mathcal{K}(\alpha)$  of convex functions of order $\alpha$  respectively. For more work in this direction, see \cite{PLD,A.W.G}.
When $\varphi$  equals $1+(2/\pi)^2(\log((1+\sqrt{z})/(1-\sqrt{z})))^2$, $\sqrt{1+z}$,  $e^z$, $1+(4/3)z+(2/3)z^{2}$, $\sin z$,  $z + \sqrt{1 + z^{2}} $ and $1+(zk+z^2/(k^2-kz))$ where $ k=\sqrt{2}+1$,  we denote  the class $\mathcal{S}^*(\varphi)$ respectively by   $\mathcal{S}_{P}$,   $\mathcal{S}_{L}^{*}$,   $\mathcal{S}_{e}^{*}$,   $\mathcal{S}_{c}^{*}$,  $\mathcal{S}_{\sin}^{*}$,  $\mathcal{S}_{\leftmoon}^{*}$, and  $\mathcal{S}_{R}^{*}$.  The class $\mathcal{S}^*_{L}$ was introduced by  Sok\'{o}l and Stankiewicz \cite{JS.JS}. For information about the other classes, we refer to the recent articles \cite{AS.VR,Kanaga,Kanika}.

The functions $f_0, f_1:\mathbb{D}\to\mathbb{C}$ defined by
\begin{align}\label{Z4} f_{0}(z)=\frac{z(1+z)^2}{(1-z)^4} \quad \text{and} \quad  f_{1}(z)=\frac{z}{(1+z)^2} \end{align}  belongs to the class $\Pi_1$ and therefore the class $\Pi_1$ is non-empty. They satisfy the required conditions with the functions $g_{0},g_{1}:\mathbb{D}\rightarrow \mathbb{C}$    defined by
\[
g_{0}(z)=\frac{z(1+z)}{(1-z)^3} \quad \text{and} \quad g_{1}(z)=\frac{z}{1-z^2};
\] indeed, we have
\[ \operatorname{Re} \frac{ f_{i}(z)}{g_{i}(z)}>0 \quad \text{and} \quad  \operatorname{Re} \frac{(1-z)^2g_{i}(z)}{z}  >0 \]
for $i=0,1$. The function $f_{0}$ is an extremal function for the radius problem that we consider. However, the function  $f_1$ is univalent, but  the function $f_0$ is not univalent as  the coefficients of the  Taylor's series  $f_{0}(z)=z + 6z^2 + 19z^3 + 44z^4   +\cdots$ do not satisfy the de Branges theorem (previously the Bieberbach conjecture).
The derivative of $f_0$ is given by
\[f_0'(z)=\frac{(1+6z+z^2)(1+z)}{(1-z)^5}.\]
Since  $f_0'(-3+2\sqrt{2})=0$ and, by Theorem~\ref{th1}~(1), the radius of starlikeness of the class $\Pi_1$ is $3-2\sqrt{2}$, it follows that the radius of univalence of this class is also $3-2\sqrt{2}\approx 0.171573$.
The other radius results for the class $\Pi_{1}$ are given in the following theorem.

\begin{theorem}\label{th1}
The following radius results hold for the class $\Pi_{1}$:
 \begin{enumerate}
  \item The $\mathcal{S}^*(\alpha)$ radius is $R_{\mathcal{S}^*(\alpha)}= (1-\alpha)/(3+\sqrt{8+\alpha^2})$, \quad $ 0\leq \alpha<1$.

  \item The $\mathcal{S}_{L}^{*}$ radius is  $R_{\mathcal{S}_{L}^{*}}= (\sqrt{2}-1)(\sqrt{10}-3)\approx 0.067217$.
  \item The $\mathcal{S}_{P}$ radius is $R_{\mathcal{S}_{P}}=(6 - \sqrt{33})/3 \approx0.0851$.
\item The $\mathcal{S}_{e}^{*}$ radius is $R_{\mathcal{S}_{e}^{*}}= (e-1)/(3 e+\sqrt{8 e^2+1}) \approx 0.1080$.
\item The $\mathcal{S}_{c}^{*}$ radius is $R_{\mathcal{S}_{c}^{*}}=  (9 - \sqrt{73})/4\approx 0.1140$.
\item The $\mathcal{S}_{\sin}^{*}$ radius is  $R_{\mathcal{S}_{\sin}^{*}}= \sin (1)/(\sqrt{9+\sin ^2(1)+2 \sin (1)}+3)  \approx0.1320$.
\item The $\mathcal{S}_{\leftmoon}^{*}$ radius is $R_{\mathcal{S}_{\leftmoon}^{*}}= 3/\sqrt{2}-\sqrt{1/2 \left(11-2 \sqrt{2}\right)}\approx0.09999$.
\item The $\mathcal{S}_{R}^{*}$ radius $R_{\mathcal{S}_{R}^{*}}= (3-2 \sqrt{5-2 \sqrt{2}})/(2 \sqrt{2}-1)\approx0.0289$.
\end{enumerate}
\end{theorem}

The functions $f_{2},f_{3}: \mathbb{D}\rightarrow \mathbb{C}$   defined by
\begin{align}\label{Z5}
f_{2}(z)=\frac{z}{1-z}   \quad \text{and}\quad  f_{3}(z)=\frac{z(1+z)^2}{(1-z )^3},
\end{align} satisfy the  conditions $|f_i(z)/g_i(z)-1|<1$ and $\operatorname{Re}((1-z)^2g_i(z))>0$ for $i=2,3$ with $g_2, g_3 : \mathbb{D}\rightarrow \mathbb{C}$   defined by
\begin{align}\label{Z5}
g_{2}(z)=\frac{z}{1-z^2}   \quad \text{and}\quad  g_{3}(z)=\frac{z(1+z)}{(1-z)^3},
\end{align}
and hence $f_{2},f_3\in \Pi_{2}$. This proves that the class   $\Pi_{2}$ is non-empty. The Taylor series $f_{3}(z)=z+5z^2+13z^3+25z^4+  \cdots$  shows that it is not  univalent. It is an extremal function for the radius problems we consider. The derivative of $f_3$ is given by
\[f_3'(z)=\frac{(1+5z)(1+z)}{(1-z)^4}.\]
Since  $f_3'(-1/5)=0$ and, by Theorem~\ref{th2}~(1), the radius of starlikeness of the class $\Pi_1$ is $1/5$, it follows that the radius of univalence of this class is also $1/5$. The other radius results for class $\Pi_2$ are given in the following theorem.

\begin{theorem}\label{th2}
The following radius results hold for the class $\Pi_{2}$:
\begin{enumerate}
\item The  $\mathcal{S}^*(\alpha)$ radius is $R_{\mathcal{S}^*(\alpha)}=2(1-\alpha)/(5+\sqrt{4 \alpha^2-4 \alpha+25})$, \quad $ 0\leq \alpha<1$.
\item The $\mathcal{S}_{L}^{*}$ radius is at least $R_{\mathcal{S}_{L}^{*}} =(\sqrt{4 \sqrt{2}+25}-5)/(2 (\sqrt{2}+2))\approx 0.0786$.
  \item The $\mathcal{S}_{p}$ radius is $R_{\mathcal{S}_{p}} =5 - 2 \sqrt{6}\approx 0.1010$.
\item The $\mathcal{S}_{e}^{*}$ radius is $\mathcal{S}_{e}^{*} =(2(e-1))/ (5 e + \sqrt{ 25 e^2 -4e+4 })\approx 0.1276$.
\item The $\mathcal{S}_{c}^{*}$ radius is $R_{\mathcal{S}_{c}^{*}}=  \left(15-\sqrt{217}\right)/2\approx 0.1345$.
\item The $\mathcal{S}_{sin}^{*}$ radius is at least $\mathcal{S}_{sin}^{*}= (\sqrt{25+4 (3+\sin (1)) \sin (1)}-5)/(2 (3+\sin (1)))\approx 0.1508$.
\item The $\mathcal{S}_{\leftmoon}^{\ast}$ radius is $R_{\mathcal{S}_{\leftmoon}^{*}}= (5-\sqrt{41-12 \sqrt{2}})/(2 \left(\sqrt{2}-1\right))\approx 0.1183$.
\item The $\mathcal{S}_{R}^{*}$ radius is $R_{\mathcal{S}_{R}^{*}}= (5-\sqrt{81-40 \sqrt{2}})/(4 \left(\sqrt{2}-1\right))\approx 0.0345$.

\end{enumerate}
\end{theorem}

It is worth to point out that $R_{\mathcal{S}_{P}^{*}}=
 R_{\mathcal{S}^{*}(1/2)}$  and $R_{\mathcal{S}_{e}^{*}}=
 R_{\mathcal{S}^{*}(1/e)}$ in both theorems.

\section{Proofs of theorems}

We need the following lemmas to prove our results.

\begin{lemma}\label{L1}\cite[Lemma 2.2, p.4]{RM.NK}
For $0<\alpha<\sqrt{2}$, let $\textbf{r}_{a}$ be given by
\[
\textbf{r}_{a} =
     \begin{cases}
            (\sqrt{1-a^2}-(1-a^2))^{\frac{1}{2}} , & 0<a\leq2\sqrt{2}/3\\
          \sqrt{2}-a, & 2\sqrt{2}/3\leq a<\sqrt{2}.
     \end{cases}
\]
Then $\left\{\omega:|\omega-a|<r_{a} \right\}\subseteq \left\{\omega:|\omega^2-1|<1 \right\}$.
\end{lemma}

\begin{lemma}\label{L2}\cite[Lemma 1, p. 321]{TN.VR}
For $a >\frac{1}{2}$, let $\textbf{r}_{a}$ be given by
\[
\textbf{r}_{a} =
     \begin{cases}
            a-\frac{1}{2} ,&  \frac{1}{2}<a \leq \frac{3}{2}\\
       \sqrt{2a-2}, &  a \geq \frac{3}{2}
     \end{cases}
\]
Then $\left\{w:| w-a|<r_{a}\right\}\subseteq\left\{w:\operatorname{Re}w>|w-1| \right\} $.
\end{lemma}

\begin{lemma}\label{L3}\cite[Lemma 2.2, p.368]{RM.SN.VR}
For $e^{-1}< a <e$, let $\textbf{r}_{a}$ be given by
\[
\textbf{r}_{a} =
     \begin{cases}
            a-e^{-1} ,&  e^{-1}< a \leq \frac{e+e^{-1}}{2}\\
       e-a,   &    \frac{e+e^{-1}}{2}\leq a \leq e.
     \end{cases}
\]
Then $\left\{w:| w-a|<\textbf{r}_{a}\right\}\subseteq\left\{ w:|\log{w}|<1\right\}=\Omega_{e}$.
\end{lemma}

\begin{lemma}\label{L4}\cite[ Lemma 2.2, p. 926]{KS.NK.VR} For $\frac{1}{3}< a <3$, let $\textbf{r}_{a}$ be given by
\[
\textbf{r}_{a} =
\begin{cases}
    a-\frac{1}{3}, & \frac{1}{3}< a \leq \frac{5}{3}  \\
     3- a, & \frac{5}{3}\leq a \leq 3.
\end{cases}
\]
Then $\left\{w:| w- a | <\textbf{r}_{a} \right\}\subseteq \Omega_{c}$, where $\Omega_{c}$ is the region bonded by the cardioid given $\left\{ x+iy :(9x^2+9y^2-18x+5)^2-16(9x^2+9y^2-6x+1)=0\right\}.$
\end{lemma}

\begin{lemma}\label{L5} \cite[Lemma 3.3, p.7]{NE.VK.SS}
For $1-\sin 1< a <1+\sin 1$, let $r_{a}=\sin 1- | a-1|$. Then $\left\{ w:|\omega- a |<\textbf{r}_{a}\right\}\subseteq \Omega_{s}$; $\Omega_{s}$ is the image of the unit disk $\mathbb{D}$ under $1+\sin z$.
\end{lemma}

\begin{lemma}\label{L6}\cite[ Lemma 2.1, p. 3]{SG.VR}. For $\sqrt{2}-1< a <\sqrt{2}+1$, let $\textbf{r}_{a}=1-|\sqrt{2}- a  |$. Then
\[\left\{w:| w- a | < \textbf{r}_{a} \right\} \subseteq \left\{w:| w^2-1| < 2 | w | \right\}.\]
\end{lemma}

\begin{lemma}\label{L7} \cite[ Lemma 2.2, p. 202]{SK.VR} For $2(\sqrt{2}-1) < a <2$, let $\textbf{r}_{a}$ be given by
\[
\textbf{r}_{a} =
     \begin{cases}
           a-2\left(\sqrt{2}-1\right) , & 2\left(\sqrt{2}-1\right)< a \leq \sqrt{2}\\
       2-a ,& \sqrt{2}\leq a \leq 2.
     \end{cases}
\]
Then $\left\{w:| w- a  | < \textbf{r}_{a}  \right\}$, where $\Omega_{r}$ is the image of the disk $\mathbb{D}$ under the function $1+(zk+z^2)/\left(k^2-kz\right)$, $k=\sqrt{2}+1$.
\end{lemma}

\begin{proof}[\bf \em Proof of Theorem~\ref{th1}]
Let the function $f\in \Pi_{1}$. Then there is a function $g:\mathbb{D}\rightarrow \mathbb{C}$   satisfying
\begin{align}\label{Z1}
{\rm {Re}}\left(\frac{f(z)}{g(z)}\right)>0 \quad  \text{and} \quad { \operatorname{Re} }\left(\frac{(1-z)^2g(z)}{z}\right)\quad \forall z\in \mathbb{D}.
\end{align}
 Define functions $p_{1}, p_{2} :\mathbb{D} \rightarrow \mathbb{C}$ as the following.
\begin{align}\label{Z2}
p_{1}(z)=\frac{(1-z)^2g(z)}{z} \quad \text{and}\quad  p_{2}(z)=\frac{f(z)}{g(z)}.
\end{align}
By using (\ref{Z1}) and (\ref{Z2}), we have  $p_{1}, p_{2}\in \mathcal{P}$, and $f(z)=zp_{1}(z)p_{2}(z)/(1-z)^2$. Then it follows that
\begin{align}\label{A4}
\frac{zf'(z)}{f(z)}=\frac{zp_{1}'(z)}{p_{2}(z)}+\frac{zp_{2}'(z)}{p_{2}(z)}+\frac{1+z}{1-z}.
\end{align}
The bilinear transformation $(1+z)/(1-z)$ maps the disk $|z|\leq r$ onto the disk
\begin{align}\label{A5}
\left|\frac{1+z}{1-z}-\frac{ 1-r^2 }{1+r^2}\right|\leq  \frac{2r}{1-r^2}.
\end{align}
 For $p\in\mathcal{P}(\alpha)$, we have
\begin{align}\label{A6}
\left|\frac{z p'(z)}{p(z)}\right|\leq\frac{2(1-\alpha)r}{(1-r)(1+(1-2\alpha)r)}, \quad  |z|\leq r.
\end{align}
By using (\ref{A4}), (\ref{A5}) and (\ref{A6}), function $f$ maps disk $|z|\leq r$ onto disk
\begin{align}\label{A}
\left| \frac{zf'(z)}{f(z)}-\frac{1+r^2}{1-r^2}\right| \leq \frac{6r}{1-r^2}.
\end{align}
From (\ref{A}), it follows that
  \begin{align}\label{Z3}
\operatorname{Re}\frac{zf'(z)}{f(z)}\geq\frac{1-6r+r^2 }{1-r^2}\geq 0,
\end{align}
for all $0\leq r \leq 3-2\sqrt{2}$.
Therefore, the function $f\in \Pi_{1}$ is starlike in $|z|\leq  3-2\sqrt{2}\approx 0.171573$. Hence, all our radii found here must be less than $3-2\sqrt{2}$.
  \begin{enumerate}[labelindent=12pt,itemindent=0em,leftmargin=!]
    \item The number $\rho= R_{\mathcal{S}^*}{(\alpha)}$, is the smallest positive root of the equation $(1+\alpha)r^2-6r+1-\alpha=0$ in $[0,1]$. For $0<r\leq \rm R_{\mathcal{S}^*(\alpha)}$, from (\ref{Z3}), it follows that
    \[\operatorname{Re}\frac{zf'(z)}{f(z)}\geq \frac{1-6r+r^2}{1-r^2}
    \geq \frac{1-6\rho+\rho^2}{1-\rho^2} = \alpha.\]
This shows that the radius of starlikeness of order $\alpha$ is at least $ R_{\mathcal{S}^*}{(\alpha)}$. To show that it is sharp, consider the function $f_0\in \Pi_{1}$ given in (\ref{Z4}). For this function $f_0$, we have
\[  \frac{zf_0'(z)}{f_0(z)} =\frac{1+6z+z^2}{1-z^2}.\]At   $z=-\rho$, we have
\[ \operatorname{Re} \frac{zf_0'(z)}{f_0(z)} =\frac{1-6\rho+\rho^2}{1-\rho^2}=\alpha,\]   proving the sharpness of the radius.

\item We can give a proof using Lemma~\ref{L1} but we give a different proof here. The number  $\rho:=R_{\mathcal{S}_{L}}$ is the smallest positive root of the equation $(1+\sqrt{2})r^2+6r+1-\sqrt{2} =0$  in interval $(0,1)$, and, from (\ref{A}), it is clear that, for $0\leq r\leq\rho$,
\begin{align}\label{D}
\left| \frac{zf'(z)}{f(z)}-1\right| \leq  \left|\frac{zf'(z)}{f(z)}-\frac{1+r^2}{1-r^2} \right|+\frac{2r^2}{1-r^2} \leq \frac{6r+2r^2}{1-r^2}\leq \frac{6\rho+\rho^2}{1-\rho^2}= \sqrt{2}-1
\end{align}
and
\begin{align}\label{E}
\left|\frac{zf'(z)}{f(z)}+1\right|\leq 2+\left|\frac{zf'(z)}{f(z)}-1\right|\leq \sqrt{2}+1.
\end{align}
Thus, from (\ref{D}) and (\ref{E}), it follows that,  for $0\leq r\leq\rho$,
\[\left| \left(\frac{zf'(z)}{f(z)}\right)^2-1\right|=\left|\frac{zf'(z)}{f(z)}-1\right|\left|\frac{zf'(z)}{f(z)}+1\right|\leq(\sqrt{2}+1)(\sqrt{2}-1)= 1.\]

For the function $f_0\in \Pi_{1}$ given in (\ref{Z4}), we have, at $z=\rho$,
\[  \frac{zf_0'(z)}{f_0(z)} = 1+\frac{6\rho+2\rho^2}{1-\rho^2}=\sqrt{2}\]
and so,   at $z=\rho$,
\[\left| \left(\frac{zf_0'(z)}{f_0(z)}\right)^2-1\right|=  1.\]
This proves the sharpness.

\item For $ \rho:=R_{ \mathcal{S}_{P}}=(6-\sqrt{33})/3$, we have
\[\frac{1}{2}< 1\leq a=\frac{1+r^2}{1-r^2} \leq\frac{1+\rho^2}{1-\rho^2}=
\frac{3 \sqrt{33}-1}{16}\approx 1.0146 <3/2.\]
Also, for $ \rho= R_{\mathcal{S}_{P}}$, we have
\[ \frac{6\rho}{(1-\rho^2)}\leq\frac{1+\rho^2}{1-\rho^2}-\frac{1}{2}\] and
the disk in (\ref{A}) for $r=\rho$ becomes
\[\left| \frac{zf'(z)}{f(z)}-a\right|=\left| \frac{zf'(z)}{f(z)}-\frac{1+\rho^2}{1-\rho^2}\right| \leq \frac{1+\rho^2}{1-\rho^2}-\frac{1}{2}=a-\frac{1}{2}.\]
By  Lemma $\ref{L2}$,  it follows that the disk in (\ref{A}) lies inside region $\Omega_{PAR}$. This proves that
the radius of parabolic starlikeness is at least $ R_{\mathcal{S}_{P}}$.

The radius  is sharp for the function $f_0\in\Pi_{1}$. At the point $z=-\rho=-R_{ \mathcal{S}_{P}}$, we have
\[\operatorname{Re}\left(\frac{zf_0'(z)}{f_0(z)}\right)=\frac{1-6\rho+\rho^2}{1-\rho^2}= \frac{1}{2}=\frac{ 6\rho-2\rho^2}{1-\rho^2}=\left|\frac{zf_0'(z)}{f_0(z)}-1\right|.\]

\item  For  $e^{-1}< a \leq \frac{e+e^{-1}}{2}$, Lemma \ref{L3} gives
\begin{align}\label{F0}
\left\{w\in\mathbb{C}:| w- a |< a -e^{-1} \right\}\subseteq \left\{w\in \mathbb{C}:| \log{w} | <1   \right\}=:\Omega_{e},
\end{align}
For $\rho= R_{\mathcal{S}^*_{e}}$, we have \[e ^{-1}< a:= \frac{1+\rho^2}{1-\rho^2}=\frac{1+9e^2}{e(1+3\sqrt{1+8e^2})} \approx 1.0236\leq   \frac{e+e^{-1}}{2}\approx 1.5430\] and, $\rho$ being smallest positive   root of the equation $(1+e)r^2-6 e r +e-1=0$,
\[\frac{6\rho}{1-\rho^2}\leq\frac{1+\rho^2}{1-\rho^2}-\frac{1}{e}= a-e^{-1}.\]Consequently,
the disk in (\ref{A}) for $r=\rho$ becomes \[
\left| \frac{zf'(z)}{f(z)}-a\right|=
\left| \frac{zf'(z)}{f(z)}-\frac{1+\rho^2}{1-\rho^2}\right| \leq \frac{1+\rho^2}{1-\rho^2}-\frac{1}{e}=a-e^{-1}.\]
By (\ref{F0}) the above disk   is inside $\Omega_{e}$ proving that the $\mathcal{S}^*_{e}$ radius for the class $\Pi_{1}$ is at least $R_{\mathcal{S}^*_{e}}$. The result is sharp for the function $f_0$ given in (\ref{Z4}). Indeed, at $z=-\rho$ where $\rho=R_{ \mathcal{S}_{e}^{*}}$, we have
\[\left|\log \left(\frac{zf_0'(z)}{f_0(z)}\right)\right|
=\left|\log\left(\frac{1-6\rho+\rho^2}{1-\rho^2}\right)\right|=1.\]

\item For $\frac{1}{3}< a\leq \frac{5}{3}$, by an application of Lemma \ref{L4}, it follows that
\begin{align}\label{F1}
\left\{w\in\mathbb{C}:| w-a |<a-\frac{1}{3}\right\}\subseteq\Omega_{c},
\end{align}
where $\Omega_{c}$ is the domain bounded by the cardioid
$\{x+iy:(9x^2+9y^2-18x+5)^2$ $-16(9x^2+9y^2-6x+1)=0\}$.   For
$\rho=R_{\mathcal{S}_{c}^{*}}$, we have
\[ \frac{1}{3}< a:=\frac{1+\rho^2}{1-\rho^2}=\frac{3\sqrt{73}-1}{24}\approx 1.0263\leq \frac{5}{3}\] and, $\rho$ being the samllest positive   root of the equation $ 2r^2-9r+1=0$,
\[ \frac{6\rho}{1-\rho^2} = \frac{1+\rho^2}{1-\rho^2}-\frac{1}{3}.\]
 Therefore, the disk in (\ref{A}) becomes
 \[
 \left| \frac{zf'(z)}{f(z)}-a\right|=
 \left| \frac{zf'(z)}{f(z)}-\frac{1+\rho^2}{1-\rho^2}\right| \leq
  = \frac{1+\rho^2}{1-\rho^2}-\frac{1}{3}=a-\frac{1}{3}
  \]
and this disk is inside $\Omega_{c}$. This shows that $\mathcal{S}_{c}^{*}$ radius is at least  $R_{\mathcal{S}_{c}^{*}}$.

For the function $f_{0}$ given in (\ref{Z4}), at $z=\rho=R_{ \mathcal{S}_{c}^{*}}$, we have
\[ \frac{zf_0'(z)}{f_0(z)}= \frac{1-6\rho+\rho^2}{1-\rho^2} =\frac{1}{3}=\varphi_c(-1)
\in\partial \varphi_c(\mathbb{D})\]
where $\varphi_c(z)=1+4z/3+2z^2/3$.

\item For   $\rho=R_{ \mathcal{S}_{\sin}^{*}}$, and $a:=(1+r^2)/(1-r^2)$, we have
    \[|a-1|=\frac{2\rho^2}{1-\rho^2}\approx 0.13199 <\sin 1\approx 0.8414. \]
and
\[\frac{6\rho}{1-\rho^2}\leq \sin 1-\frac{2\rho^2}{1-\rho^2}. \]
The   disk in (\ref{A}) for $r=\rho$ becomes
\[\left|\frac{zf'(z)}{f(z)} -a\right| =\left|\frac{zf'(z)}{f(z)} - \frac{1+\rho^2}{1-\rho^2} \right|
\leq \sin 1-\frac{2\rho^2}{1-\rho^2}=\sin 1-|1-a|.\]
Lemma $\ref{L5}$ shows that   the disk in (\ref{A}) is inside $\Omega_{s}$
where $\Omega_{s}=:\varphi_s(\mathbb{D})$ is the image of the unit disk $\mathbb{D}$ under the mapping $\varphi_s(z)=1+\sin z$. This proves that the
$\mathcal{S}_{\sin}^{*}$ radius is at least $R_{ \mathcal{S}_{\sin}^{*}}$.
For the function $f_{0}$   given in (\ref{Z4}), with $\rho=R_{ \mathcal{S}_{\sin}^{*}}$, we have
\[  \left(\frac{zf'(z)}{f(z)}\right) =  \frac{1+6\rho+\rho^2}{1-\rho^2} =1+\sin 1\in\varphi_s(1)\in \partial \varphi_s(\mathbb{D}).\]

\item  For $\rho= R_{\mathcal{S}_{\leftmoon}^{*}}$, we have
\[a:=\frac{1+\rho^2}{1-\rho^2} \approx 1.0202 \in (\sqrt{2}-1,\sqrt{2}+1)\]
and
\[\frac{1-6\rho+\rho^2}{1-\rho^2} =\sqrt{2}-1.\]
The disk in (\ref{A}) becomes
\[ \left|\frac{zf'(z)}{f(z)} -a \right|\leq 1-| \sqrt{2}-a | \]
 and by Lemma~\ref{L6} it  lies inside $\left\{w : | w^2-1|<2| w |\right\}$.
This shows that $\mathcal{S}_{\leftmoon}^{*}$ radius is at least  $R_{\mathcal{S}_{\leftmoon}^{*}}$. The sharpness follows as  the function $f_{0}$ defined in (\ref{Z4}) satisfies, at  $z=\rho=R_{ \mathcal{S}_{\leftmoon}^{*}}$,
\[
\begin{split}  \left|\left(\frac{zf_0'(z)}{f_0(z)}\right)^2-1\right|&=
\left|\left(\frac{1-6\rho+\rho^2}{1-\rho^2}\right)^2-1\right|=2(\sqrt{2}-1)\\
& =2  \frac{1-6\rho+\rho^2}{1-\rho^2}
=2 \left|\frac{zf'_0(z)}{f_0(z)} \right|.
\end{split} \]

\item For $\rho=R_{\mathcal{S}^*_{R}}$, we have
\[ 2 ( \sqrt{2}-1 )< a:=\frac{1+\rho^2}{1-\rho^2}\approx 1.00167 \leq \sqrt{2}<2,
\] and
\[\frac{1-6\rho+\rho^2}{1-\rho^2}= 2-2\sqrt{2}.\] The disk
 (\ref{A}) becomes
\begin{align}\label{F4}
\left|\frac{zf'(z)}{f(z)}-a\right|< a-2 ( \sqrt{2}-1).
\end{align}
By Lemma $\ref{L7}$, this disk lies inside  the domain $\Omega_{r}$. This proves that $ \mathcal{S}_{R}^{*} $ radius  is at least $R_{ \mathcal{S}_{R}^{*}}$.

To prove sharpness, consider the function $f_{0}\in \Pi_{1}$ given in (\ref{Z4}). At $z=-\rho=-R_{ \mathcal{S}_{R}^{*}}$, we have
 \[ \frac{zf'(z)}{f(z)} =  \frac{1-6\rho+\rho^2}{1-\rho^2}
 =2(\sqrt{2}-1)=\varphi_r(-1)\in\partial\varphi_r(\mathbb{D})\]
where $\varphi_r(z)=1+(kz+z^2)/(k^2-kz)$, $k=\sqrt{2}+1$. \qedhere
\end{enumerate}
\end{proof}

\begin{proof}[\bf \em Proof of Theorem~\ref{th2}]
Since $|w-1|<1$ is equivalent to $\operatorname{Re} (1/w)>1/2$, the condition  $| f(z)/g(z) -1|<1$ is the same as the condition ${\operatorname{Re}  }\left(g(z)/f(z) \right)>1/2$. Let the function $f \in \Pi_{2}$. Let  $g:\mathbb{D}\rightarrow\mathbb{C}$ be chosen such that \begin{align}\label{A1}
\operatorname{Re}\left(\frac{g(z)}{f(z)} \right)>\frac{1}{2} \quad  \text{and} \quad  \operatorname{Re}\left(\frac{(1-z)^2}{z}g(z) \right).
\end{align}
Define $p_{1},p_{2}: \mathbb{D}\rightarrow \mathbb{C}$ as
\begin{align}\label{B2}
p_{1}(z)=\frac{(1-z)^2}{z}g(z),\quad p_{2}(z)=\frac{g(z)}{f(z)}.
\end{align}
 From $(\ref{A1})$  and $(\ref{B2})$ it follows that $p_{1} \in\mathcal{P}$, and $p_{2}\in \mathcal{P}(1/2)$, and $f(z)=z/(1-z)^2p_{1}(z)/p_{2}(z)$. A calculation shows that  \begin{align}\label{X1}
\frac{zf'(z)}{f(z)}= \frac{zp_{1}'(z)}{p_{1}(z)}-\frac{zp_{2}'(z)}{p_{2}(z)}+\frac{1+z}{1-z}.
\end{align}
 The bilinear transformation $\omega=(1+z)/(1-z)$ maps the disk $| z|\leq r $ onto disk
\begin{align}\label{X2}
\left| \frac{1+z}{1-z}-\frac{1+r^2}{1-r^2} \right| \leq\frac{2r}{1-r^2}.
\end{align}
Recall that for $p\in \mathcal{P}(\alpha)$, we have
 \begin{align}\label{X3}
\left| \frac{z p'(z)}{p(z)} \right|\leq \frac{2(1-\alpha)r}{(1-r)(1+(1-2\alpha)r)}, \quad  | z|\leq r.
\end{align}
Using  $(\ref{X2})$ and $(\ref{X3})$ in $(\ref{X1})$, we get
\begin{align}\label{D4}
\left|\frac{zf'(z)}{f(z)}-\frac{1+r^2}{1-r^2}\right| \leq \frac{5r+r^2}{1-r^2}.
\end{align}
From (\ref{D4}), it follows that
\begin{align}\label{A26}
\operatorname{Re}\left(\frac{zf'(z)}{f(z)}\right) \geq \frac{1-5r}{1-r^2}\geq 0
\end{align}
for $0\leq r \leq 1/5$. For the function $f_3$ given in \eqref{Z5}, we have
\[\frac{zf_3'(z)}{f_3(z)}=\frac{1+5z}{1-z^2} =0\]
for $z=-1/5$. Thus, the radius of starlikeness of the class $\Pi_{2}$ is $1/5$. All radius  values to be  computed here will be less than $1/5$.

\begin{enumerate}[labelindent=12pt,itemindent=0em,leftmargin=!]
  \item The number  $\rho:= R_{\mathcal{S}^*}(\alpha)$ is the smallest positive root of the equation $\alpha r^2-5 r+1-\alpha=0$. For $0<r\leq R_{\mathcal{S}^*}(\alpha)$, from (\ref{A26}), we have
\[\operatorname{Re}\left( \frac{zf'(z)}{f(z)}\right)\geq  \frac{1-5r}{1-r^2}\geq \frac{1-5\rho}{1-\rho^2}=\alpha.\]
For the  function $f_{2}\in \Pi_{2}$ given in (\ref{Z5}), we have, at $z=-\rho=-R_{\mathcal{S}^*}(\alpha)$,
 \[\frac{zf_3'(z)}{f_3(z)}=\frac{1-5\rho}{1-\rho^2}=   \alpha.\]
This proves that the radius of starlikeness of order $\alpha$ is $R_{\mathcal{S}^*}(\alpha)$.

  \item From (\ref{D4}),  it  follows that
\begin{align}\label{Z6}
\left| \frac{zf'(z)}{f(z)}-1\right| \leq  \left|\frac{zf'(z)}{f(z)}-\frac{1+r^2}{1-r^2}\right| +\frac{2r^2}{1-r^2}
\leq \frac{5r+3r^2}{1-r^2}.
\end{align}
The number $\rho=R_{\mathcal{S}^*_{L}}$, is the   positive root of the equation $5r+3r^2-(1-r^2)(\sqrt{2}-1) = 0$. For $0<r\leq \rho= R_{\mathcal{S}_{L}^{*}}$,  we have
\begin{align}\label{F5}
  \frac{5r+3r^2}{1-r^2}\leq \frac{5\rho+3\rho^2}{1-\rho^2} =\sqrt{2}-1.
\end{align}
Therefore, by (\ref{Z6}), (\ref{F5}), and for $0<r\leq \rho= R_{\mathcal{S}^*_{L}}$, it follows that
\begin{align}\label{Z7}
\left|\frac{zf'(z)}{f(z)} -1\right|   \leq \sqrt{2}-1,
\end{align}and
\begin{align}\label{Z8}
\left|\frac{zf'(z)}{f(z)}+1 \right|\leq\sqrt{2}+1.
\end{align}
The last two inequalities immediately yields
\[\left|\left(\frac{zf'(z)}{f(z)}\right)^2-1 \right|\leq\left| \frac{zf'(z)}{f(z)} +1\right|\, \left| \frac{zf'(z)}{f(z)} -1\right|\leq (\sqrt{2}+1)(\sqrt{2}-1)=1.\] This proves that $ \mathcal{S}^*_{L} $ is at least $R_{\mathcal{S}^*_{L}}$.

  \item
  For $0\leq r\leq \rho:=R_{ \mathcal{S}_{P}}=5-2\sqrt{6}$, we have
for
\[\frac{1}{2}< 1\leq a=\frac{1+\rho^2}{1-\rho^2}=
\frac{5 \sqrt{6}}{12}<3/2\] and, $\rho$ being the smallest positive  root of the equation $r^2-10r+1=0$,
\[\frac{5\rho+\rho^2}{(1-\rho^2)} \leq \frac{1+\rho^2}{1-\rho^2}-\frac{1}{2}.\]
The disk in (\ref{D4})  becomes 
\begin{align*}
\left|\frac{zf'(z)}{f(z)}-\frac{1+\rho^2}{1-\rho^2}\right| \leq \frac{1+\rho^2}{1-\rho^2}-\frac{1}{2}.
\end{align*}
By Lemma \ref{L2}, the  disk in (\ref{D4}) is  inside the region $\Omega_{PAR}$. Thus,  the radius of parabolic starlikeness of the class $\Pi_{2}$ is at least $R_{\mathcal{S}_{p}}$. 

For the function $f_3$ given in (\ref{Z5}) at $z=-\rho$ where $\rho=R_{\mathcal{S}_{p}}$, we have
\[\operatorname{Re}\left(\frac{zf_3'(z)}{f_3(z)}\right)=\frac{1-5\rho}{1-\rho^2}
=\frac{5\rho-\rho^2}{1-\rho^2}=\left|\frac{zf_3'(z)}{f_3(z)}-1\right|.\]

\item For $\rho=R_{\mathcal{S}_{e}^{*}}$, we have  $1/e < a :=(1+\rho^2)/(1-\rho^2)\approx 1.0331 \leq (e+e^{-1})/2$ and
\[\frac{5\rho+\rho^2}{1-\rho^2}= \frac{1+\rho^2}{1-\rho^2}-\frac{1}{e}.\]
The disk in (\ref{D4})  becomes
\begin{align*}
\left|\frac{zf'(z)}{f(z)}-\frac{1+\rho^2}{1-\rho^2}\right| \leq  \frac{1+\rho^2}{1-\rho^2}-\frac{1}{e}.
\end{align*} 
By Lemma~\ref{L3}, this disk  is inside the region $\Omega_{e}$, proving that $\mathcal{S}_{e}^{*}$ radius is at least $R_{\mathcal{S}_{e}^{*}}$. 

The result is sharp for  the function $f_{3}$ given in (\ref{Z5}). For this function,  we have, at  $z= -\rho $ where $\rho=R_{\mathcal{S}_{e}^{*}}$,
\[\left|\log\left(\frac{zf_3'(z)}{f_3(z)}\right)\right|=
\left|\log\left(\frac{1-5\rho}{1-\rho^2}\right)\right|
=\left|\log(e^{-1})\right|=1.\]

\item  For $\rho=R_{\mathcal{S}_{c}^{*}}$, we have  $1/3 < a :=(1+\rho^2)/(1-\rho^2)=\frac{1}{72}(1+5\sqrt{217})\approx 1.03686 \leq 5/2$ and, $\rho$ being the smallest positive root of $r^2-15r+2=0$, 
\[\frac{5\rho+\rho^2}{1-\rho^2}= \frac{1+\rho^2}{1-\rho^2}-\frac{1}{3}.\]
The disk in (\ref{D4})  becomes
\begin{align*}
\left|\frac{zf'(z)}{f(z)}-\frac{1+\rho^2}{1-\rho^2}\right| \leq  \frac{1+\rho^2}{1-\rho^2}-\frac{1}{3}.
\end{align*}
By Lemma~\ref{L3}, this disk  is inside the region $\Omega_{c}$, proving that $\mathcal{S}_{c}^{*}$ radius is at least $R_{\mathcal{S}_{c}^{*}}$.

The radius is sharp for  the function $f_{3}$ given in (\ref{Z5}). At   $z=-\rho $ where $\rho=R_{\mathcal{S}_{c}^{*}}$, we have
\[ \frac{zf'_3(z)}{f_3(z)} = \frac{1-5\rho}{1-\rho^2}  =\frac{1}{3}=\varphi_c(-1)
\in\partial \varphi_c(\mathbb{D})\]
where $\varphi_c(z)=1+4z/3+2z^2/3$.

 \item For   $\rho=R_{ \mathcal{S}_{\sin}^{*}}$, and $a:=(1+\rho^2)/(1-\rho^2)$, we have
    \[|a-1|=\frac{2\rho^2}{1-\rho^2}\approx 0.0465396 <\sin 1\approx 0.8414. \]
and
\[\frac{5\rho+\rho^2}{1-\rho^2}\leq \sin 1-\frac{2\rho^2}{1-\rho^2}. \]
The   disk in (\ref{A}) for $r=\rho$ becomes
\[\left|\frac{zf'(z)}{f(z)} -a\right| =\left|\frac{zf'(z)}{f(z)} - \frac{1+\rho^2}{1-\rho^2} \right|
\leq \sin 1-\frac{2\rho^2}{1-\rho^2}=\sin 1-|1-a|.\]
Lemma \ref{L5} shows that   the disk in \eqref{D4} is inside $\Omega_{s}$
where $\Omega_{s}=:\varphi_s(\mathbb{D})$ is the image of the unit disk $\mathbb{D}$ under the mapping $\varphi_s(z)=1+\sin z$. This proves that the
$\mathcal{S}_{\sin}^{*}$ radius is at least $R_{ \mathcal{S}_{\sin}^{*}}$.

 \item For $\rho= R_{\mathcal{S}_{\leftmoon}^{*}}$, we have
\[a:=\frac{1+\rho^2}{1-\rho^2} \approx 1.02839 \in (\sqrt{2}-1,\sqrt{2}+1)\]
and
\[\frac{5\rho+\rho^2}{1-\rho^2} = \frac{1+\rho^2}{1-\rho^2}+1-\sqrt{2}.\]
The disk in (\ref{D4}) becomes
\[ \left|\frac{zf'(z)}{f(z)} -a \right|\leq 1-| \sqrt{2}-a | \]
 and by Lemma~\ref{L6} it  lies inside $\left\{w : | w^2-1|<2| w |\right\}$.
This shows that $\mathcal{S}_{\leftmoon}^{*}$ radius is at least  $R_{\mathcal{S}_{\leftmoon}^{*}}$. The sharpness follows as  the function $f_3$ defined in (\ref{Z4}) satisfies, at  $z=\rho=R_{ \mathcal{S}_{\leftmoon}^{*}}$,
\[
\begin{split}  \left|\left(\frac{zf_3'(z)}{f_3(z)}\right)^2-1\right|&=
\left|\left(\frac{1-5\rho}{1-\rho^2}\right)^2-1\right|=2(\sqrt{2}-1)\\
& =2  \frac{1-5\rho}{1-\rho^2}
=2 \left|\frac{zf'_3(z)}{f_3(z)} \right|.
\end{split} \]

\item For $\rho=R_{\mathcal{S}^*_{R}}$, we have
\[ 2 ( \sqrt{2}-1 )< a:=\frac{1+\rho^2}{1-\rho^2}\approx 1.00238 \leq \sqrt{2}<2,
\] and
\[\frac{5\rho+\rho^2}{1-\rho^2}= \frac{1+\rho^2}{1-\rho^2}-2(\sqrt{2}-1).\] The disk  (\ref{D4}) becomes
\begin{align*}
\left|\frac{zf'(z)}{f(z)}-a\right|< a-2 ( \sqrt{2}-1).
\end{align*}
By Lemma $\ref{L7}$, this disk lies inside  the domain $\Omega_{r}$. This proves that $ \mathcal{S}_{R}^{*} $ radius  is at least $R_{ \mathcal{S}_{R}^{*}}$.

To prove sharpness, consider the function $f_3\in \Pi_{1}$ given in (\ref{Z4}). At $z=-\rho=-R_{ \mathcal{S}_{R}^{*}}$, we have
 \[ \frac{zf_3'(z)}{f_3(z)} =  \frac{1-5\rho }{1-\rho^2}
 =2(\sqrt{2}-1)=\varphi_r(-1)\in\partial\varphi_r(\mathbb{D})\]
where $\varphi_r(z)=1+(kz+z^2)/(k^2-kz)$, $k=\sqrt{2}+1$. \qedhere 
\end{enumerate}
\end{proof}

We have only obtained   lower bounds for the   $\mathcal{S}_{L}^{*}$ and  $\mathcal{S}_{sin}^{*}$ radii for the class $\Pi_2$ and we believe the bounds are sharp but unable to prove it. 


\end{document}